\documentclass[a4paper,12pt]{amsart}
\usepackage{amsmath,amsfonts,amssymb,latexsym}
\usepackage{hyperref}
\newcommand{\RR}{\mathbb{R}}
\newcommand{\CC}{\mathbb{C}}

\newcommand{\NN}{\mathbb{N}}

\newcommand{\Bo}{\mathcal{B}}
\newcommand{\Oo}{\mathcal{O}}
\newcommand{\dist}{ {\rm dist} }

\newcommand{\nint}[1]{\rlap{~~~~\;---}\int\limits_{#1}}

\newcommand{\rint}[1]{\rlap{\,\,\,\,\,\,\,~~~~~~~~\,\;---}\int\limits_{#1}}

\newcommand{\RE}{ {\rm Re \,} }

\newtheorem{Th}{Theorem}
\newtheorem{Lem}{Lemma}
\newtheorem{Cor}{Corollary}
\newtheorem{Prop}{Proposition}

\theoremstyle{definition}
\newtheorem{Df}{Definition}

\theoremstyle{remark}
\newtheorem{Rem}{Remark}
\newtheorem{Ex}{Example}
\begin{document}
\title[Summable solutions of some PDEs]{Summable solutions of some partial differential equations and generalised integral means}
\author{S{\l}awomir Michalik}
\address{Faculty of Mathematics and Natural Sciences,
College of Science\\
Cardinal Stefan Wyszy\'nski University\\
W\'oycickiego 1/3,
01-938 Warszawa, Poland}
\email{s.michalik@uksw.edu.pl}
\urladdr{www.impan.pl/~slawek}
\subjclass[2010]{35B05, 35C10, 35E15, 40G10}
\keywords{$k$-summability, formal power series, linear partial differential equations, generalised integral means, Pizzetti's formula}
\date{}
\begin{abstract}
We describe partial differential operators for which we can construct generalised integral means satisfying Pizzetti-type formulas.
Using these formulas we give a new characterisation of summability of formal power series solutions to some multidimensional partial differential
equations in terms of holomorphic properties of generalised integral means of the Cauchy data.
\end{abstract}
\maketitle

\section{Introduction}
We consider the initial value problem for a multidimensional linear partial differential equation with constant coefficients 
\begin{gather}
 \label{eq:1.1}
(\partial_t - P(\partial_z))u = 0,\qquad u(0,z)=\varphi(z),
\end{gather}
where $t\in \mathbb{C}$, $z\in\mathbb{C}^n$, $P(\partial_z)\in\mathbb{C}[\partial_z]$ is a differential operator of order greater than $1$ with 
complex coefficients and $\varphi$ is holomorphic in a complex neighbourhood of the origin. The unique formal power series solution $\widehat{u}$
of (\ref{eq:1.1}) is in general divergent. So, it is natural to ask about sufficient and necessary conditions 
(expressed in terms of the Cauchy datum $\varphi$) under which the formal solution $\widehat{u}$ is convergent or, more generally, summable.

If the spatial variable $z$ is one-dimensional then, by the general theory introduced by Balser \cite{B5}, 
one can find the sufficient condition for summability of $\widehat{u}$ in
terms of the analytic continuation property of the Cauchy datum $\varphi$ (see \cite[Theorem 5]{B5}). Moreover, in this case, 
due to a certain symmetry between the variables $t$ and $z$, this sufficient condition is also necessary
(see \cite[Theorem 4]{Mic5} and \cite{Mic10}). 

But if the spatial variable $z$ is multidimensional then the characterisation of analytic and summable solutions of (\ref{eq:1.1})
in terms of $\varphi$ is much more complicated. 

Such characterisation were found in the special case $P(\partial_z)=\Delta_z=\sum_{k=1}^n\partial_{z_k}^2$, when (\ref{eq:1.1}) is the Cauchy problem
for the multidimensional heat equation. Namely, in this case {\L}ysik \cite{L2} proved that $\widehat{u}$ is convergent if and only if the integral
mean of $\varphi$ over the closed ball $B(x,r)$ or the sphere $S(x,r)$, as a function of the radius $r$, extends to an entire function of
exponential order at most $2$. Moreover, the author \cite{Mic9} showed that $\widehat{u}$ is $1$-summable in a direction $d$  if and only if
the integral mean of $\varphi$ over the closed ball $B(x,r)$ or the sphere $S(x,r)$, can be analytically continued to infinity in some sectors
bisected by $d/2$ and $\pi+d/2$ with respect to $r$, and this continuation is of exponential order at most $2$ as $r$ tends to infinity.

These characterisations are based on the Pizzetti formulas which give the expansions of the integral means of $\varphi$  in terms of the
radius $r$ with coefficients depending on the iterated Laplacians of $\varphi$ (see Proposition \ref{pr:pizzetti}).

Such expansion for the
surface integral mean in $\RR^3$ was established already in 1909 by Pizzetti \cite{P}.
His result was extended to $\RR^n$ by Nicolesco \cite{N}. Another proof of classical Pizzetti formula in $\RR^n$, which 
is based on so called generating functions of the solid and spherical means, was given recently by Toma \cite{To}.

The Pizzetti formula has been also generalised by Gray and Willmore \cite{G-W} to arbitrary Riemannian manifolds,
by Da Lio and Rodino \cite{Lio-Ro} to the case of the heat operator, and by Bonfiglioli \cite{Bo1,Bo2} to the case of some
subelliptic operators such as the Kohn Laplace operator on the Heisenberg group and its generalisation to orthogonal sub-Laplacians on $H$-type groups.

In the paper we are concerned with Pizzetti-type formulas which express a given generalised integral mean in $\RR^n$ (or $\CC^n$)
as a series in the radius $r$
whose coefficients depend on the iterated operator $P(\partial_z)$,
where $P(\partial_z)$ is a given partial differential operator with constant coefficients.
Observe that our approach to the generalisation of the Pizzetti formula is different from \cite{Lio-Ro} and \cite{Bo1,Bo2},
where the coefficients of the expansions are not given directly by the iterated heat operator or by the iterated subelliptic operators. 

Generalised integral means and Pizzetti-type formulas are the main tools used in the paper. Using them, we extend the results \cite{L2}
and \cite{Mic9} to more general multidimensional partial differential operators $P(\partial_z)$. 

We show that if a generalised integral mean of $\varphi$ satisfies a Pizzetti-type formula for the operator $P(\partial_z)$ then we are
able to characterise convergent (and under some additional conditions also summable) solutions of (\ref{eq:1.1}) in terms of the generalised
integral means (Theorem \ref{th:sum}).
For this reason it is important to describe such operators $P(\partial_z)$,
for which we can find generalised integral means satisfying a Pizzetti-type formula.

In the paper we construct
such generalised integral means for
every homogeneous operator with real coefficients of order $1$ (Proposition \ref{pr:first}) and for every elliptic homogeneous operator of order $2$
with real
coefficients (Proposition \ref{pr:elliptic}). As a corollary we obtain a characterisation of analytic and $1$-summable solutions
of (\ref{eq:1.1}) in terms of generalised integral means in the case when $P(\partial_z)=\sum_{i,j=1}^na_{ij}\partial^2_{z_iz_j}$
is a homogeneous elliptic operator of order $2$ with real
coefficients (Theorem \ref{th:elliptic}). We also prove that it is impossible to find a generalised integral mean
satisfying a Pizzetti-type formula for any homogeneous (Theorem \ref{th:homo}) or quasi-homogeneous (Theorem \ref{th:quasi}) operator of 
order $p>2$.

We also extend the notion of generalised integral means to the complex case. Then we are able to construct
complex generalised integral means satisfying Pizzetti-type formulas for $P(\partial_z)=Q^s(\partial_z)$, where $Q(\partial_z)$
is a homogeneous operator of order $p\leq 2$ and $s\in\NN$. In particular, as corollaries we obtain characterisations of analytic
and summable solutions of (\ref{eq:1.1}) in terms of complex generalised integral means in the case when
$P(\partial_{z})=\partial_{z_1}^2-\sum_{k=2}^n\partial_{z_k}^2=\square_z$ is the wave operator (Corollary \ref{co:wave}),
$P(\partial_z)=\Delta_z^s$ is the $s$-Laplace operator (Corollary \ref{co:laplace}),
$P(\partial_z)=(\sum_{k=1}^na_k\partial_{z_k})^s$ (Corollary \ref{co:first}) and 
$P(\partial_z)=  (\sum_{i,j=1}^na_{ij}\partial^2_{z_iz_j})^s$ (Corollary \ref{co:second}).

\section{Notation}
Throughout the paper $B(x,r)$ ($S(x,r)$, respectively) denotes the real closed ball (sphere, respectively) with centre at $x\in\RR^n$ and
radius $r>0$. Moreover, the complex disc in $\CC^n$ with centre at
the origin and radius $r>0$ is denoted by $D^n_r:=\{z\in\CC^n:\ |z|< r\}$. To simplify notation, we write $D_r$ instead of $D^1_r$. 
If the radius $r$ is not essential, then we denote it briefly by $D^n$ (resp. $D$).
\par
The \emph{Pochhammer symbol} is defined for non-negative integers $k$ and complex numbers $a$ as $(a)_0:=1$ and $(a)_k:=a(a+1)\cdots(a+k-1)$
for $k\in\NN$.
\par
\par
\emph{A sector in a direction $d\in\RR$ with an opening $\varepsilon>0$}
in the universal covering space $\widetilde{\CC\setminus\{0\}}$ of $\CC\setminus\{0\}$ is defined by
\[
S_d(\varepsilon):=\{z\in\widetilde{\CC\setminus\{0\}}:\ z=re^{i\theta},\
d-\varepsilon/2<\theta<d+\varepsilon/2,\ r>0\}.
\]
Moreover, if the value of opening angle $\varepsilon$ is not essential, then we denote it briefly by $S_d$.
\par
Analogously, by a \emph{disc-sector in a direction $d\in\RR$ with an opening $\varepsilon>0$ and radius $r_1>0$}
we mean a domain $\widehat{S}_d(\varepsilon;r_1):=S_d(\varepsilon)\cup D_{r_1}$. If the values of $\varepsilon$ and $r_1$
are not essential, we write it $\widehat{S}_d$ for brevity (i.e. $\widehat{S}_d=S_d\cup D$).
\par
By $\mathcal{O}(G)$ we understand the space of holomorphic functions on a domain $G\subseteq\CC^n$.
\par
The space of formal power series
$ \widehat{u}(t,z)=\sum_{j=0}^{\infty}u_j(z) t^j$ with $u_j(z)\in\Oo(D^n)$ is denoted by $\Oo(D^n)[[t]]$ or shortly by $\Oo[[t]]$.
We use the ``hat'' notation (i.e. $\widehat{u}$) to denote the formal power series.
If the formal power series $\widehat{u}$ is convergent, we denote its sum by $u$.

\section{Moment functions and k-summability}
   We recall the notion of moment methods introduced by Balser \cite{B2} and next we use them to
   define the Gevrey order and the $k$-summability. For more details we refer the reader to \cite{B2}.

\begin{Df}
A function $u(t,z)\in\Oo(\widehat{S}_d(\varepsilon;r_1)\times D^n_r)$
is of \emph{exponential growth of order at most $k>0$ as $t\to\infty$
in $\widehat{S}_d(\varepsilon;r_1)$} if for every $\widetilde{r}\in(0,r)$, $\widetilde{r}_1\in(0,r_1)$ and every
$\widetilde{\varepsilon}\in(0,\varepsilon)$ there exist
$A,B<\infty$ such that
\begin{gather*}
\max\limits_{|z|\leq \widetilde{r}}|u(t,z)| \leq Ae^{B|t|^k}
\quad \textrm{for} \quad t\in \widehat{S}_d(\widetilde{\varepsilon};\widetilde{r}_1).
\end{gather*}
The space of such functions is denoted by $\Oo^k(\widehat{S}_d(\varepsilon;r_1)\times D^n_r)$. If the values of $\varepsilon$, $r_1$ and $r$
are not essential, we write it $\Oo^k(\widehat{S}_d\times D^n)$ for brevity.
\end{Df}
   
   \begin{Df}[see {\cite[Section 5.5]{B2}}]
    \label{df:moment}
    A pair of functions $e_m$ and $E_m$ is said to be \emph{kernel functions of order $k$} ($k>1/2$) if
    they have the following properties:
   \begin{enumerate}
    \item[1.] $e_m\in\Oo(S_0(\pi/k))$, $e_m(z)/z$ is integrable at the origin, $e_m(x)\in\RR_+$ for $x\in\RR_+$ and
     $e_m$ is exponentially flat of order $k$ in $S_0(\pi/k)$ (i.e. for every $\varepsilon > 0$ there exist $A,B > 0$
     such that $|e_m(z)|\leq A e^{-(|z|/B)^k}$ for $z\in S_0(\pi/k-\varepsilon)$).
    \item[2.] $E_m\in\Oo^{k}(\CC)$ and $E_m(1/z)/z$ is integrable at the origin in $S_{\pi}(2\pi-\pi/k)$.
    \item[3.] The connection between $e_m$ and $E_m$ is given by the \emph{corresponding moment function
    $m$ of order $1/k$} as follows.
     The function $m$ is defined in terms of $e_m$ by
     \begin{gather}
      \label{eq:e_m}
      m(u):=\int_0^{\infty}x^{u-1} e_m(x)dx \quad \textrm{for} \quad \RE u \geq 0
     \end{gather}
     and the kernel function $E_m$ has the power series expansion
     \begin{gather}
      \label{eq:E_m}
      E_m(z)=\sum_{j=0}^{\infty}\frac{z^j}{m(j)} \quad  \textrm{for} \quad z\in\CC.
     \end{gather}
   \end{enumerate}
   \end{Df}
   
    Observe that in case $k\leq 1/2$ the set $S_{\pi}(2\pi-\pi/k)$ is ill-defined,
    so the second property in Definition \ref{df:moment} can not be satisfied. It means that we
    must define the kernel functions of order $k\leq 1/2$ and the corresponding moment functions
    in another way.
    
    \begin{Df}[see {\cite[Section 5.6]{B2}}]
     \label{df:small}
     A function $e_m$ is called \emph{a kernel function of order $k>0$} if we
     can find a pair of kernel functions $e_{\widetilde{m}}$ and $E_{\widetilde{m}}$ of
     order $pk>1/2$ (for some $p\in\NN$) so that
     \begin{gather*}
      e_m(z)=e_{\widetilde{m}}(z^{1/p})/p \quad \textrm{for} \quad z\in S(0,\pi/k).
     \end{gather*}
     For the kernel function $e_m$ of order $k>0$ we define the
     \emph{corresponding moment function $m$ of order $1/k>0$} by (\ref{eq:e_m}) and
     the \emph{kernel function $E_m$ of order $k>0$} by (\ref{eq:E_m}).
    \end{Df}
    
    \begin{Rem}
     \label{re:m_tilde}
     Observe that by Definitions \ref{df:moment} and \ref{df:small} we have
     \begin{eqnarray*}
      m(u)=\widetilde{m}(pu) & \textrm{and} &
      E_m(z)=\sum_{j=0}^{\infty}\frac{z^j}{m(j)}=\sum_{j=0}^{\infty}\frac{z^j}{\widetilde{m}(jp)}.
     \end{eqnarray*}
     Hence if $\widetilde{m}$ is a moment function of order $s>0$ then $m(u):=\widetilde{m}(pu)$
     is a moment function of order $ps$.
    \end{Rem}

\begin{Ex}
\label{ex:functions}
 For any $a\geq 0$, $b\geq 1$ and $k>0$ we can construct the following examples of kernel functions $e_m$ and 
 $E_m$ of 
orders $k>0$ with the corresponding moment function $m$ of order $1/k$ satisfying Definition \ref{df:moment}
or \ref{df:small}:
 \begin{itemize}
  \item $e_m(z)=akz^{bk}e^{-z^k}$,
  \item $m(u)=a\Gamma(b+u/k)$, where $\Gamma$ is the gamma function,
  \item $E_m(z)=\frac{1}{a}\sum_{j=0}^{\infty}\frac{z^j}{\Gamma(b+j/k)}$.
\end{itemize}
In particular for $a=b=1$ we get the kernel functions and the corresponding moment function, which are used  
in the classical theory of $k$-summability.
\begin{itemize}
    \item $e_m(z)=kz^ke^{-z^k}$,
    \item $m(u)=\Gamma(1+u/k)$,
    \item $E_m(z)=\sum_{j=0}^{\infty}\frac{z^j}{\Gamma(1+j/k)}=:\mathbf{E}_{1/k}(z)$, where $\mathbf{E}_{1/k}$ is the
    Mittag-Leffler function of index $1/k$.
\end{itemize}
For any $s>0$ we will denote by $\Gamma_s$ the function $\Gamma_s(u):=\Gamma(1+su)$,
which is the crucial example of a moment function of order $s$.
\end{Ex}

By \cite[Theorems 31 and 32]{B2}  we obtain the following proposition, which allows us to construct new moment functions.
\begin{Prop}
 \label{pr:moments}
 Let $m_1$, $m_2$ be moment functions of orders $s_1,s_2>0$ respectively. Then
 \begin{itemize}
  \item $m_1m_2$ is a moment function of order $s_1+s_2$,
  \item $m_1/m_2$ is a moment function of order $s_1-s_2$ for $s_1>s_2$.
 \end{itemize}
\end{Prop}
In particular we have
\begin{Cor}
 \label{co:moment}
 We assume that $s>1$. Then $m_1$ is a moment function of order $s$ if and only if $m_2=m_1/\Gamma_1$ is a moment function of order $s-1$.
\end{Cor}

Every moment function $m$ of order $s>0$ has the same growth as $\Gamma_s$. Precisely speaking,
we have 
\begin{Prop}[see {\cite[Section 5.5]{B2}}]
  \label{pr:order}
  If $m$ is a moment function of order $s>0$
  then there exist constants $c,C>0$ such that
    \begin{gather}
     \label{eq:order}
     c^j\Gamma_s(j)\leq m(j) \leq C^j\Gamma_s(j) \quad \textrm{for every} \quad j\in\NN.
    \end{gather}
  \end{Prop}

More generally we say
\begin{Df}
 If $s>0$ and a function $m$ satisfies (\ref{eq:order}) then
 $m$ is called a \emph{function of order $s$}.  
\end{Df}

We use moment functions to define moment Borel transforms, the Gevrey order and the Borel summability.
\begin{Df}
 Let $m$ be a moment function of order $s>0$ (or, more generally, a function of order $s>0$). Then the linear operator $\Bo_{m}\colon \Oo[[t]]\to\Oo[[t]]$ defined by
 \[
  \Bo_{m}\big(\sum_{j=0}^{\infty}u_jt^{j}\big):= \sum_{j=0}^{\infty}\frac{u_j}{m(j)}t^{j}
 \]
 is called the \emph{$m$-moment Borel transform of order $s$}.
\end{Df}

We define the Gevrey order of formal power series as follows
   \begin{Df}
    \label{df:summab}
    Let $s>0$. Then
    $\widehat{u}\in\Oo(D^n)[[t]]$ is called a \emph{formal power series of Gevrey order $s$} if
    there exists a disc $D\subset\CC$ with centre at the origin such that
    $\Bo_{\Gamma_s}\widehat{u}\in\Oo(D\times D^n)$. The space of formal power series of Gevrey 
    order $s$ is denoted by $\Oo(D^n)[[t]]_s$.
   \end{Df}

\begin{Rem}
 \label{re:moment}
 By Proposition \ref{pr:order}, we may replace $\Gamma_s$ in Definition
 \ref{df:summab} by any function $m$ of order $s$.
\end{Rem}

Now we are ready to define the $k$-summability of formal power series (see Balser \cite{B2})
\begin{Df}
\label{df:summable}
Let $k>0$ and $d\in\RR$. Then $\widehat{u}\in\Oo(D^n)[[t]]$ is called
\emph{$k$-summable in a direction $d$} if there exists a disc-sector $\widehat{S}_d$ in a direction $d$ such that
$\Bo_{\Gamma_{1/k}}\widehat{u}(t,z)\in\Oo^k(\widehat{S}_d\times D^n)$.
\end{Df}

We also define a class of functions of order $s$, which contains moment functions of order $s$.
\begin{Df}
 If there exists a moment function $m_2$ of order $s>0$ and
 polynomials $p_1(j)$, $p_2(j)$ and $c>0$ such that a function $m_1$ of order $s$ satisfies
 \begin{gather}
  \label{eq:similar_to}
  p_1(j)m_1(j)=p_2(j)m_2(j)c^j \quad \textrm{for every} \quad j\in\NN
 \end{gather}
 then $m_1$ is called a \emph{generalised moment function of order $s$}.
\end{Df}

The importance of the assumption (\ref{eq:similar_to}) comes from
\begin{Prop}
  \label{pr:5}
  We assume that $m_1(j)$ and $m_2(j)$ are sequences satisfying (\ref{eq:similar_to}), $k>0$, $d\in\RR$
  and $\widehat{F}_i(t)=\sum_{j=0}^{\infty}\frac{a_j}{m_i(j)}t^j$ for $i=1,2$.
  Then $F_1(t)\in\Oo^k(\widehat{S}_d)$ if and only if $F_2(t)\in\Oo^k(\widehat{S}_d)$.
 \end{Prop}
 \begin{proof}
  Since $m_1(j)$ and $m_2(j)$ satisfy (\ref{eq:similar_to}), 
  there exist polynomials $p_1(j)$, $p_2(j)$ and $c>0$ such that
 \begin{eqnarray*}
 \frac{p_1(j)}{m_2(j)}=\frac{p_2(j)}{m_1(j)}c^j & \textrm{for every} & j\in\NN.
 \end{eqnarray*}
 Hence $m_1(j)$ and $m_2(j)$ have the same order and consequently
 also $\widehat{F}_1(t)$ and $\widehat{F}_2(t)$ have the same Gevrey order, so in particular
 $\widehat{F}_1(t)\in\CC[[t]]_0$ if and only if $\widehat{F}_2(t)\in\CC[[t]]_0$.
 Moreover, if $\widehat{F}_i(t)\in\CC[[t]]_0$ ($i=1,2$) then their sums $F_i(t)$ are well defined and satisfy
 $$
 p_1(t\partial_t)F_2(t)=\sum_{j=0}^{\infty}\frac{p_1(j)a_j}{m_2(j)}t^j=\sum_{j=0}^{\infty}\frac{p_2(j)a_jc^j}{m_1(j)}t^j
 = p_2(t\partial_t)F_1(ct).
 $$
 It means that $F_1(t)\in\Oo^k(\widehat{S}_d)$ if and only if $F_2(t)\in\Oo^k(\widehat{S}_d)$.
 \end{proof}

By \cite[Proposition 13]{B2} and Definition \ref{df:summable} we may also characterise convergent series $\widehat{u}$ in terms of
$\Bo_{\Gamma_{1/k}}\widehat{u}$ as follows
\begin{Prop}
 \label{pr:analytic}
 Let $k>0$ and $\widehat{u}\in\Oo(D^n)[[t]]$. Then $\widehat{u}(t,z)$ converges for sufficiently small $|t|$ if and only if
 $\Bo_{\Gamma_{1/k}}\widehat{u}(t,z)\in\Oo^k(\CC\times D^n)$.
\end{Prop}

\begin{Rem}
 \label{re:general_moment}
 By the general theory of moment summability (see \cite[Section 6.5 and Theorem 38]{B2}) and by Proposition \ref{pr:5}, we may replace
 $\Gamma_{1/k}$  in Definition \ref{df:summable} by any 
 moment function $m$ of order $1/k$ or even by any
 generalised moment function of order $1/k$.
\end{Rem}

 By Proposition \ref{pr:analytic} and by Remarks \ref{re:moment}--\ref{re:general_moment} we conclude that
 \begin{Prop}
  \label{pr:similar}
  Let $k>0$, $d\in\RR$, $m$ be a function of order $1/k$ and $\widehat{u}=\sum_{j=0}^{\infty}u_j(z)t^j\in\Oo(D^n)[[t]]$.
  Then $\widehat{u}$ is convergent for sufficiently small $|t|$ if and only if 
  $\sum_{j=0}^{\infty}\frac{u_j(z)}{m(j)}t^{j} \in\Oo^{k}(\CC\times D^n)$.
  
  Moreover, if additionally $m$ is a generalised moment function of order $1/k$, then $\widehat{u}$ is $k$-summable in a direction $d$ if and only if
  there exists a disc-sector $\widehat{S}_{d}$ in a direction $d$ such that
    $\sum_{j=0}^{\infty}\frac{u_j(z)}{m(j)}t^{j} \in\Oo^{k}(\widehat{S}_d\times D^n)$.
 \end{Prop}

\section{Generalised integral means}
 In this section we introduce the notion of generalised integral means. 
 To this end we take 
\begin{Df}
 \label{df:generalised}
   Let $\mu$ be a finite complex Borel measure supported in the closed ball $B(0,R)$
(for some $R>0$) in $\RR^n$ of total mass $1$ (i.e. $\int_{\RR^n}\,d\mu(y)=1$). Moreover we assume that $\varphi$ is a continuous function
  on a domain $\Omega\subset \RR^n$, $x\in\Omega$ and $0<r<\dist(x,\partial\Omega)/R$. Then the \emph{generalised integral mean
  $M_{\mu}(\varphi;x,r)$
  of $\varphi$} is defined by
  \begin{equation}
   \label{eq:generalised}
    M_{\mu}(\varphi;x,r):=\int_{\RR^n}\varphi(x+ry)\,d\mu(y).
  \end{equation}
\end{Df}

\begin{Rem}
 Let us assume that
 \begin{gather}
  \label{eq:means}
   \mu_B(y)=\frac{dy}{\alpha(n)} \quad \textrm{and} \quad \mu_S(y)=\frac{dS(y)}{n\alpha(n)},
 \end{gather}
where $\alpha(n)=\pi^{n/2}/\Gamma(1+n/2)$ is the volume of the $n$-dimensional unit ball $B(0,1)$,
$dy$ is the Lebesgue measure on $B(0,1)$ and $dS(y)$ is
 the surface measure on $S(0,1)$. Then 
 $$M_{\mu_B}(\varphi;x,r)=\int_{B(0,1)}\varphi(x+yr)\,d\mu_B(y)=\nint{B(0,1)}\varphi(x+yr)\,dy$$
 and
 $$M_{\mu_S}(\varphi;x,r)=\int_{S(0,1)}\varphi(x+yr)\,d\mu_S(y)=\nint{S(0,1)}\varphi(x+yr)\,dS(y)$$
 are respectively the solid and spherical means of $\varphi$.
\end{Rem}
\begin{Rem}
 We can extend Definition \ref{df:generalised} to the complex case replacing in (\ref{eq:generalised}) the real variables $x,r$
 by the complex ones $z,t$. In this way we define by (\ref{eq:generalised}) the generalised integral mean $M_{\mu}(\varphi;z,t)$
 for $z\in\CC^n$ and $t\in\CC$.
\end{Rem}
The crucial role in our considerations is played by the extensions of the Pizzetti formulas, which hold for some generalised integral means
$M_{\mu}(\varphi;z,t)$. Precisely
\begin{Df}
Let $\varphi(z)\in\Oo(D^n)$ and $P(\partial_z)\in\CC[\partial_z]$
be a homogeneous partial differential operator of order $p$ with constant coefficients.
We say that a generalised integral mean $M_{\mu}(\varphi;z,t)$ satisfies a \emph{Pizzetti-type formula for the operator $P(\partial_z)$} if
\begin{equation}
\label{eq:star}
M_{\mu}(\varphi;z,t)=\sum_{j=0}^{\infty}\frac{P^j(\partial_z)\varphi(z)}{m(j)}t^{pj}
\end{equation}
for some function $m$ of order $p>0$ with $m(0)=1$.
\end{Df}

Our aim is to describe such generalised integral means $M_{\mu}(\varphi;z,t)$ which satisfy Pizzetti-type formulas for some operators
$P(\partial_z)$.

\begin{Rem}
 \label{re:integral_mean}
 Observe that a generalised integral mean $M_{\mu}(\varphi;z,t)$ satisfying the Pizzetti-type formula (\ref{eq:star})
 for the operator $P(\partial_z)$ has the following natural properties:
\begin{enumerate}
\item[a)] $\lim\limits_{t\to 0}M_{\mu}(\varphi;z,t)=M_{\mu}(\varphi;z,0)=\varphi(z)$.
\item[b)] The mean-value property for $P(\partial_z)$: If $P(\partial_z)\varphi(z)=0$ then
  $\varphi(z)=M_{\mu}(\varphi;z,t)$ for any $z\in D^n$ and $t\in\CC$.
\item[c)] Converse to the mean-value property for $P(\partial_z)$:
 If $\varphi(z)=M_{\mu}(\varphi;z,t)$ for any $z\in D^n$ and $t>0$ then $P(\partial_z)\varphi(z)=0$.
\end{enumerate}
\end{Rem}

\section{Formal power series solutions}
We will use generalised integral means satisfying Pizzetti-type formulas to characterise summable formal power series solutions
of the Cauchy problem
\begin{eqnarray}
 \label{eq:P}
 (\partial_t - P(\partial_z))u=0,&&u(0,z)=\varphi(z)\in\Oo(D^n),
\end{eqnarray}
where $t\in\CC$, $z\in\CC^n$ and $P(\partial_z)\in\CC[\partial_z]$ is a homogeneous partial differential operator of order $p>1$
with constant coefficients.
The unique formal power solution of (\ref{eq:P}) is given by
\begin{equation*}
\widehat{u}(t,z)=\sum_{j=0}^{\infty}\frac{P^j(\partial_z)\varphi(z)}{j!}t^j\in\Oo(D^n_r)[[t]].
\end{equation*}
Since $P(\partial_z)$ is an operator of order $p$, for every $r_0\in(0,r)$ there exist $A,B<\infty$ such that
\begin{gather*}
 \sup_{|z|<r_0}\Big|\frac{P^j(\partial_z)\varphi(z)}{j!}\Big|\leq AB^j(j!)^{p-1} \quad \textrm{for every} \quad j\in\NN_0.
\end{gather*}
It means that $\widehat{u}(t,z)$ is a formal series of Gevrey order $p-1$. So, it is natural to ask about
$\frac{1}{p-1}$-summability
of $\widehat{u}(t,z)$. By Proposition \ref{pr:similar} and Corollary \ref{co:moment} we conclude
\begin{Th}
  \label{th:sum}
  Let $P(\partial_z)\in\CC[\partial_z]$ be a homogeneous partial differential operator of order $p>1$ with constant coefficients and
  $M_{\mu}(\varphi;z,t)$ be a generalised integral mean satisfying a Pizzetti-type formula (\ref{eq:star}) for the operator $P(\partial_z)$.
  Then the formal solution of (\ref{eq:P}) is convergent for sufficiently small $|t|$
  if and only if the generalised integral mean $M_{\mu}(\varphi;z,t)$ satisfies 
\begin{gather*}
M_{\mu}(\varphi;z,t) \in\Oo^{\frac{p}{p-1}}(D^n\times\CC) \quad \textrm{for} \quad k=0,\dots,p-1.
\end{gather*}
\par
Moreover, if we additionally assume that $m$ is a generalised moment function of order $p$
then the formal solution
of (\ref{eq:P}) is
  $\frac{1}{p-1}$-summable in a direction $d\in\RR$ if and only if there exist
  disc-sectors $\widehat{S}_{\frac{d+2k\pi}{p}}$ in the directions $\frac{d+2k\pi}{p}$ (for $k=0,\dots,p-1$) 
  such that the generalised integral mean $M_{\mu}(\varphi;z,t)$ satisfies 
\begin{gather*}
M_{\mu}(\varphi;z,t) \in\Oo^{\frac{p}{p-1}}(D^n\times\widehat{S}_{\frac{d+2k\pi}{p}})
\quad \textrm{for} \quad k=0,\dots,p-1.
\end{gather*}
\end{Th}

\section{First and second order homogeneous operators}
In this section we describe generalised integral means satisfying
Pizzetti-type formulas (\ref{eq:star}) for the homogeneous operators of first and second order.

First we assume that $P(\partial_z)$ is a general homogeneous
first order operator with real coefficients. For such operators we have
\begin{Prop}
 \label{pr:first}
 Let $P(\partial_z)=\sum_{k=1}^n a_k\partial_{z_k}$ with $a=(a_1,\dots,a_n)\in \RR^n$ and $\mu_{\delta,a}(y)=\delta(y-a)$,
 where $\delta$ is the Dirac delta.
 Then the generalised integral mean  $M_{\mu_{\delta,a}}(\varphi;z,t)$ satisfies a Pizzetti-type formula
 \begin{equation}
 \label{eq:first}
  M_{\mu_{\delta,a}}(\varphi;z,t)=\varphi(z+at)=\sum_{j=0}^{\infty}\frac{P^j(\partial_z)\varphi(z)}{j!}t^j.
 \end{equation}
\end{Prop}
\begin{proof}
By the definition
$$M_{\mu_{\delta,a}}(\varphi;z,t)=\int_{\RR^n}\varphi(z+ty)\,d\mu_{\delta,a}(y)=\varphi(z+at).$$
Moreover, by the Taylor formula for the function $t\mapsto\varphi(z+at)$ we conclude that
 $$
 \varphi(z+at)=\sum_{j=0}^{\infty}\frac{\frac{d^j}{dt^j}\varphi(z+at)|_{t=0}}{j!}t^j=
 \sum_{j=0}^{\infty}\frac{P^j(\partial_z)\varphi(z)}{j!}t^j
 $$
 and (\ref{eq:first}) holds.
 \end{proof}
 
 For the Laplace operator $P(\partial_z)=\sum_{k=1}^n\partial^2_{z_k}=\Delta_z$ we get the classical formulas
 introduced for the first time in dimension $n=3$ by Pizzetti \cite{P}.
 Namely, we have
\begin{Prop}[{\cite[Theorem 3.1]{L2}}]
\label{pr:pizzetti}
 We assume that the measures $\mu_B(y)$ and $\mu_S(y)$ are defined by (\ref{eq:means}) and the sequences of numbers $m_B$ and $m_S$ are given by
 \begin{gather}
  \label{eq:numbers}
  m_B(j)=4^j(n/2+1)_j j! \quad\textrm{and}\quad m_S(j)=4^j(n/2)_jj!.
 \end{gather}
Then the solid $M_{\mu_B}(\varphi;z,t)$ and spherical $M_{\mu_S}(\varphi;z,t)$ means
 satisfy the Pizzetti formulas
 $$
 M_{\mu_B}(\varphi;z,t)=\nint{B(0,1)}\varphi(z+ty)\,dy=\sum_{j=0}^{\infty}\frac{\Delta_z^j\varphi(z)}{m_B(j)}t^{2j}
 $$
 and
 $$
 M_{\mu_S}(\varphi;z,t)=\nint{S(0,1)}\varphi(z+ty)\,dS(y)=\sum_{j=0}^{\infty}\frac{\Delta_z^j\varphi(z)}{m_S(j)}t^{2j}.
 $$
\end{Prop}
Moreover, using the Pizzetti formulas we conclude
 \begin{Prop}[{\cite[Theorems 4.5 and 4.6]{L2}} and {\cite[Theorem 4.1]{Mic9}}]
 \label{pr:laplace}
  Let $d\in\RR$ and $\widehat{u}\in\Oo[[t]]$ be the formal solution of the $n$-dimensional complex heat equation
  \begin{eqnarray*}
(\partial_t - \Delta_z)u(t,z) = 0,&& u(0,z)=\varphi(z)\in\mathcal{O}(D^n).
\end{eqnarray*}
Then:
\begin{itemize}
 \item $u\in\Oo(D^{n+1})$ if and only if $M_{\mu_B}(\varphi;z,t)\in\Oo^2(D^n\times \CC)$, and if and only if
 $M_{\mu_S}(\varphi;z,t)\in\Oo^2(D^n\times \CC)$,
 \item $\widehat{u}$ is $1$-summable in the direction $d$ if and only if
$M_{\mu_B}(\varphi;z,t)\in\Oo^2(D^n\times (\widehat{S}_{d/2}\cup\widehat{S}_{d/2+\pi}))$, and if and only if
$M_{\mu_S}(\varphi;z,t)\in\Oo^2(D^n\times (\widehat{S}_{d/2}\cup\widehat{S}_{d/2+\pi}))$.
\end{itemize}
\end{Prop}
\begin{Rem}
 Observe that
\[
  m_B(j)=\left\{
  \begin{array}{lll}
     p_{B1}(j)(2j)! & \textrm{for} & \textrm{odd}\ n\\
     p_{B2}(j)4^j (j!)^2 & \textrm{for} & \textrm{even}\ n
  \end{array}
  \right.
\]
with polynomials $p_{B1}(j)=\frac{1}{n!!}(2j+n)(2j+n-2)\cdots(2j+1)$
and $p_{B2}(j)=\frac{1}{n!!}(2j+n)(2j+n-2)\cdots(2j+2)$. Analogously 
\[
  m_S(j)=\left\{
  \begin{array}{lll}
     p_{S1}(j)(2j)! & \textrm{for} & \textrm{odd}\ n\\
     p_{S2}(j)4^j (j!)^2 & \textrm{for} & \textrm{even}\ n
  \end{array}
  \right.
\]
with polynomials $p_{S1}(j)=\frac{1}{(n-2)!!}(2j+n-2)(2j+n-4)\cdots(2j+1)$
and $p_{S2}(j)=\frac{1}{(n-2)!!}(2j+n-2)(2j+n-4)\cdots(2j+2)$. 
Moreover by Example \ref{ex:functions} and Proposition \ref{pr:moments}
we see that $\Gamma(1+2u)$ and $\Gamma^2(1+u)$ are moment functions of order $2$. Hence
$m_B(u)$ and $m_S(u)$ are generalised moment functions of order $2$ and Proposition \ref{pr:laplace} is an easy consequence
of Theorem \ref{th:sum} and Proposition \ref{pr:pizzetti}.
\end{Rem}

We may generalise the above results to the case when $P(\partial_z)=\sum_{i,j=1}^n a_{ij}\partial^2_{z_i z_j}$ is a homogeneous elliptic
operator of order 
$2$ with real coefficients. So, we may assume that $A=(a_{ij})_{i,j=1}^n$ is a symmetric nonsingular positive-definite real matrix.
Then there exists an orthogonal real matrix $C$ such that
 $\Lambda=C^TAC$ is a diagonal matrix with positive entries on the main diagonal.
 Then
 \begin{eqnarray*}
   P(\partial_z)\varphi(z)=\Delta_w\tilde{\varphi}(w) & \textrm{for}\ w = \Lambda^{-1/2}C^Tz &
   \textrm{and}\ \tilde{\varphi}(w)=\varphi(C\Lambda^{1/2}w).
 \end{eqnarray*}
  
 Applying the above change of variables to Proposition \ref{pr:pizzetti} we get
 \begin{Prop}
 \label{pr:elliptic}
 Let us assume that 
 \begin{gather}
  \label{eq:measure_el}
  \mu_B(y)=\frac{d(\Lambda^{-1/2}C^Ty)}{\alpha(n)}=\frac{dy}{|C\Lambda^{1/2}|\alpha(n)}\quad\textrm{and}\quad
  \mu_S(y)=\frac{dS(\Lambda^{-1/2}C^Ty)}{n\alpha(n)},
 \end{gather}
 where $|C\Lambda^{1/2}|$ is the determinant of the matrix $C\Lambda^{1/2}$, 
 $\alpha(n)=\pi^{n/2}/\Gamma(1+n/2)$ is the volume of the $n$-dimensional unit ball $B(0,1)$,
$dy$ is the Lebesgue measure on $\RR^n$ and $dS(y)$ is the surface measure on $S(0,1)$. 
 Then the generalised solid $M_{\mu_B}(\varphi;z,t)$ and spherical $M_{\mu_S}(\varphi;z,t)$ means
 satisfy Pizzetti-type formulas
 $$
 M_{\mu_B}(\varphi;z,t)=\rint{C\Lambda^{1/2}(B(0,1))}\varphi(z+ty)\,dy=\nint{B(0,1)}\varphi(z+tC\Lambda^{1/2}y)\,dy
 =\sum_{j=0}^{\infty}\frac{P(\partial_z)^j\varphi(z)}{m_B(j)}t^{2j}
 $$
 and
 $$
 M_{\mu_S}(\varphi;z,t)=\nint{S(0,1)}\varphi(z+tC\Lambda^{1/2}y)\,dS(y)=\sum_{j=0}^{\infty}\frac{P(\partial_z)^j\varphi(z)}{m_S(j)}t^{2j},
 $$
 where $m_B(j)$ and $m_S(j)$ are defined by (\ref{eq:numbers}).
\end{Prop}
Hence, as in Proposition \ref{pr:laplace} we conclude
  \begin{Th}
  \label{th:elliptic}
  Let $d\in\RR$, $P(\partial_z)=\sum_{i,j=1}^n a_{ij}\partial^2_{z_i z_j}$ be a homogeneous elliptic
operator of order 
$2$ with real coefficients and $\widehat{u}$ be the formal solution of
  \begin{eqnarray*}
(\partial_t - P(\partial_z))u(t,z) = 0,&& u(0,z)=\varphi(z)\in\mathcal{O}(D^n),
\end{eqnarray*}
Moreover, let $M_{\mu_B}(\varphi;z,t)$ and $M_{\mu_S}(\varphi;z,t)$ be generalised integral means with measures defined by (\ref{eq:measure_el}).
Then: 
 \begin{itemize}
  \item $u\in\Oo(D^{n+1})$ if and only if $M_{\mu_B}(\varphi;z,t)\in\Oo^2(D^n\times \CC)$, and if and only if
  $M_{\mu_S}(\varphi;z,t)\in\Oo^2(D^n\times \CC)$,
  \item $\widehat{u}$ is $1$-summable in the direction $d$ if and only if
  $M_{\mu_B}(\varphi;z,t)\in\Oo^2(D^n\times (\widehat{S}_{d/2}\cup\widehat{S}_{d/2+\pi}))$, and if and only if
  $M_{\mu_S}(\varphi;z,t)\in\Oo^2(D^n\times (\widehat{S}_{d/2}\cup\widehat{S}_{d/2+\pi}))$.
 \end{itemize}
\end{Th} 

\section{Higher order homogeneous operators}
In the previous section we have found the generalised integral means satisfying Pizzetti-type formulas (\ref{eq:star})
for first order homogeneous operators and second order elliptic homogeneous operators.
In this section we will consider the homogeneous operators $P(\partial_z)$ of order $p>2$. We will show that
for such operators it is impossible to find any generalised integral means satisfying (\ref{eq:star}). To this end we use the
following results:
\begin{Lem}[Zalcman theorem {\cite[Theorem 1]{Z}}]
 \label{le:Zalcman}
   Let $\mu$ be a finite complex Borel measure of compact support on $\RR^n$,
   $$
   F(\zeta)=\int_{\RR^n}e^{-i(\zeta\cdot y)}\,d\mu(y)
   $$
   be the Fourier-Laplace transform of $\mu$ and $\varphi\in\Oo(D^n)$.
   Then for each $z\in D^n$ we have
   \begin{equation*}
   M_{\mu}(\varphi;z,t)=\int_{\RR^n}\varphi(z+yt)\,d\mu(y)=F(it\partial_z)\varphi(z)
   \end{equation*}
   for all $t\in\CC$ for which the integral exists and the right-hand side converges. 
  \end{Lem}

  \begin{Lem}[Paley-Wiener-Schwartz theorem {\cite[Theorem 1.7.7]{H}}] 
  \label{le:P-W-S}
  $F(z)\in\Oo(\CC^n)$ is a Fourier-Laplace transform of a distribution
   $v$ with compact support (${\rm supp}\,v\subset B(0,r)$) if and only if
   there exist $C>0$ and $N\in\NN$ such that
   \begin{equation}
    \label{eq:P-W-S}
     |F(\zeta)|\leq C(1+|\zeta|)^N e^{r|{\rm Im}\,(\zeta)|}.
   \end{equation}
   \end{Lem}
  
  \begin{Lem}[Phragm\'en-Lindel\"of theorem, {\cite{P-L}}]
   \label{le:P-L}
   Let $d\in\RR$, $\alpha<\pi$ and $F\in\Oo^1(S_d(\alpha))$. We assume that
   $|F(\zeta)|< M $ for $\zeta\in\partial S_d(\alpha)$. Then $|F(\zeta)|< M $ for $\zeta\in S_d(\alpha)$.
  \end{Lem} 

Now we are ready to state the main result of this section:  
  \begin{Th}
  \label{th:homo}
   If $p>2$, $P(\partial_z)$ is a homogeneous operator of order $p$ and $m$ is a function of order $p$ 
   then 
   \begin{equation*}
    \sum_{j=0}^{\infty}\frac{P^j(\partial_z)\varphi(z)}{m(j)}t^{pj}
   \end{equation*}
   is not a generalised integral mean for any Borel measure $\mu$ of compact support in $\RR^n$.
  \end{Th}
  \begin{proof}
  First observe that for every homogeneous operator $P(\partial_z)$ of order $p$
  $$
  F(it\partial_z)\varphi(z)=\sum_{j=0}^{\infty}\frac{P^j(\partial_z)\varphi(z)}{m(j)}t^{pj}
    $$
    if and only if
    \begin{equation}
     \label{eq:F}
     F(\zeta)=\sum_{j=0}^{\infty}\frac{P^j(-i\zeta)}{m(j)}.
    \end{equation}
   Hence by Lemmas \ref{le:Zalcman} and \ref{le:P-W-S} it is sufficient to show that the function
   $F(\zeta)$ defined by (\ref{eq:F}) does not satisfy the estimate (\ref{eq:P-W-S}).
   \par
   To the contrary, let us suppose that $F(\zeta)$ satisfies (\ref{eq:P-W-S}).
   We introduce an auxiliary function
   \begin{gather}
    \label{eq:def_f}
    f(t):=F(t,\dots,t) \qquad \textrm{for} \qquad t\in\CC. 
   \end{gather}  
   By (\ref{eq:P-W-S}), if $\arg t = 0$ then the function $f(t)$
   has a polynomial growth at infinity, i.e. there exist $C>0$ and $N\in\NN$ such that
   \begin{equation}
    \label{eq:polynomial} 
    |f(t)| \leq C(1+|t|)^N\quad 
   \end{equation} 
   for every $t\in\CC$ with $\arg t=0$. Since $P(\cdot)$ is a homogeneous polynomial of order $p$, $f(t)$ is
   invariant under the rotation of the angle $\frac{2\pi}{p}$ (i.e. $f(t)=f(te^{\frac{2\pi i}{p}})$ for every $t\in\CC$).
   It means that $f(t)$ satisfies
   (\ref{eq:polynomial}) also for every $t\in\CC$ such that $\arg t=\frac{2k\pi}{p}$ for some $k\in\{0,\dots,p-1\}$.
   \par
   Next we fix the sector $\tilde{S}_k:=S_{\frac{(2k+1)\pi}{p}}(\frac{2\pi}{p})$ and $a_k\in\CC\setminus\overline{\tilde{S}_k}$ for
   $k=0,\dots,p-1$. Then the function
   \begin{gather*}
   f_k(t):=\frac{f(t)}{(t-a_k)^N}\quad\textrm{for}\quad t\in\overline{\tilde{S}_k}
   \end{gather*}
   satisfies the following conditions:
   \begin{enumerate}
    \item[a)] there exists $M_k>0$ such that $|f_k(t)|\leq M_k$ for every $t\in\partial \tilde{S}_k$,
    \item[b)] $f_k(t)\in O^1(\tilde{S}_k)$.
   \end{enumerate}
   Moreover, the opening of $\tilde{S}_k$ is equal to $\frac{2\pi}{p}$ and $\frac{2\pi}{p}<\pi$ for $p>2$. It means that the assumptions
   of Lemma \ref{le:P-L}
   are satisfied. Consequently, by Lemma \ref{le:P-L}
   \begin{gather*}
    |f_k(t)|\leq M_k \quad \textrm{for every} \quad t\in \tilde{S}_k \quad \textrm{and} \quad k=0,\dots,p-1.
   \end{gather*}
   Hence there exist $\tilde{M}_k<\infty$ ($k=0,\dots,p-1$) satisfying 
   \begin{gather*}
   |f(t)|\leq \tilde{M}_k(1+|t|)^N \quad \textrm{for every} \quad t\in \overline{\tilde{S}_k} \quad \textrm{and} \quad k=0,\dots,p-1.
   \end{gather*}
   Taking $M=\sup\{\tilde{M}_0,\dots,\tilde{M}_{p-1}\}$ we conclude that
   \begin{eqnarray}
    \label{eq:growth}
    |f(t)|\leq M(1+|t|)^N &\textrm{for every}& t\in\CC.
   \end{eqnarray}
   On the other hand, if $P(-i,\dots,-i)=re^{i\phi}$, $m(j)\leq A B^j (pj)!$ and
   \begin{gather*}
    t=e^{-\frac{i\phi}{p}}R \qquad \textrm{for} \qquad R>0,
   \end{gather*}
   then
   \begin{eqnarray*}
    f(t)& = &\sum_{j=0}^{\infty}\frac{e^{-i\phi j} P^j(-i,\dots,-i)R^{pj}}{m(j)}=\sum_{j=0}^{\infty}\frac{(R\sqrt[p]{r})^{pj}}{m(j)}\\
    &\geq& \frac{1}{A} 
      \sum_{j=0}^{\infty}\frac{(R\sqrt[p]{r/B})^{pj}}{(pj)!}\geq ae^{bR}
    \end{eqnarray*}
    for some $a,b>0$. Hence $f(t)$ is of exponential growth as $R\to\infty$, contrary to (\ref{eq:growth}). It means that
    $F(\zeta)$ does not satisfy (\ref{eq:P-W-S}), which completes the proof.
    \end{proof}
    
\section{Quasi-homogeneous operators of higher order}
We extend the results of the previous section to the case of quasi-homogeneous operators using the ideas of Pokrovskii \cite{Po}.
\begin{Df}
  Let $N=(N_1,\dots,N_n)\in\NN^n$ and $p\in\NN_0$. A polynomial $P(\zeta)$, $\zeta\in\CC^n$, (respectively an operator $P(\partial_z)$)
  is said to be
  \emph{quasi-homogeneous of type $N$ and of order $p$} if $P(\zeta)=\sum a_k\zeta^k\not\equiv 0$, where the sum is taken over the set of
  all multi-indices $k=(k_1,\dots,k_n)\in\NN_0^n$ such that $|kN|=k_1N_1+\cdots+k_nN_n=p$ and $\zeta^k:=\zeta_1^{k_1}\cdots\zeta_n^{k_n}$.
\end{Df}
Observe that every homogeneous operator is quasi-homogeneous of type $(1,\dots,1)$.
More interesting example is given by the heat operator $P(\partial_z)=\partial_{z_1}-\partial_{z_2}^2-\cdots-\partial_{z_n}^2$,
which is a quasi-homogeneous operator of type $(2,1,\dots,1)$ and of order $2$.
  
For quasi-homogeneous operators it is convenient to consider the following $N$-version of generalised integral means
\begin{Df}
   Let $\mu$ be a finite complex Borel measure supported in the closed ball $B(0,R)$
(for some $R>0$) in $\RR^n$ of total mass $1$ (i.e. $\int_{\RR^n}\,d\mu(y)=1$). Moreover we assume that $\varphi$ is a continuous function
  on a domain $\Omega\subset \RR^n$, $x\in\Omega$ and $r>0$ is so small that $x+r^Ny\in\Omega$ for very $y\in B(0,R)$,
  where $r^Ny=r^{N_1}y_1+\cdots+r^{N_n}y_n$.
  Then the \emph{$N$-generalised integral mean $M^N_{\mu}(\varphi;x,r)$
  of $\varphi$} is defined by
  \begin{equation*}
    M_{\mu}^N(\varphi;x,r):=\int_{\RR^n}\varphi(x+r^Ny)\,d\mu(y).
  \end{equation*}
\end{Df}

For the $N$-generalised integral means we have the following version of Lemma \ref{le:Zalcman}
\begin{Lem}[see \cite{Po}]
   \label{le:Po}
   Let $\mu$ be a finite complex Borel measure of compact support on $\RR^n$,
   $$
   F(\zeta)=\int_{\RR^n}e^{-i(\zeta\cdot y)}\,d\mu(y)
   $$
   be the Fourier-Laplace transform of $\mu$ and $\varphi\in\Oo(D^n)$.
   Then for each $z\in D^n$ we have
   \begin{equation*}
   M^N_{\mu}(\varphi;z,t)=\int_{\RR^n}\varphi(z+t^Ny)\,d\mu(y)=F(it^{N}\partial_z)\varphi(z)
   \end{equation*}
   for all $t\in\CC$ for which the integral exists and the right-hand side converges. 
\end{Lem}

So, we can extend Theorem \ref{th:homo} to the quasi-homogeneous operators as follows
\begin{Th}
  \label{th:quasi}
   If $p>2$, $P(\partial_z)$ is a quasi-homogeneous operator of type $N$ and of order $p$ and $m$ is a function of order $p$ 
   then 
   \begin{equation*}
    \sum_{j=0}^{\infty}\frac{P^j(\partial_z)\varphi(z)}{m(j)}t^{pj}
   \end{equation*}
   is not an $N$-generalised integral mean for any Borel measure $\mu$ of compact support in $\RR^n$.
\end{Th}
\begin{proof}
  First, similarly to the proof of Theorem \ref{th:homo}, we observe that for quasi-homogeneous operators $P(\partial_z)$ of type $N$
  and of order $p$
  $$
  F(it^N\partial_z)\varphi(z)=\sum_{j=0}^{\infty}\frac{P^j(\partial_z)\varphi(z)}{m(j)}t^{pj}
    $$
    if and only if
  \begin{equation}
    \label{eq:quasi}
    F(\zeta)=\sum_{j=0}^{\infty}\frac{P^j(-i\zeta)}{m(j)},
  \end{equation}
  so by Lemmas \ref{le:P-W-S} and \ref{le:Po} it is sufficient to show that $F(\zeta)$ given by (\ref{eq:quasi}) does not satisfy
  (\ref{eq:P-W-S}). To this end we repeat the proof of Theorem \ref{th:homo} replacing  (\ref{eq:def_f}) by
  \begin{gather*}
    f(t):=F(t^N)=F(t^{N_1},\dots,t^{N_n}) \qquad \textrm{for} \qquad t\in\CC.
  \end{gather*}
\end{proof}

\section{Complex generalised integral means}
In this section we extend the definition of generalised integral means to the complex case. Namely we have
\begin{Df}
   Let $\mu$ be a finite complex Borel measure supported in the closed disc $\overline{D_R^n}$
(for some $R>0$) in $\CC^n$ of total mass $1$ (i.e. $\int_{\CC^n}\,d\mu(y)=1$). Moreover we assume that $\varphi$ is a continuous function
  on a domain $G\subset \CC^n$, $z\in G$ and $0<r<\dist(z,\partial G)/R$. Then the \emph{complex generalised integral mean
  $M_{\mu}(\varphi;z,r)$
  of $\varphi$} is defined by
  \begin{equation*}
    M_{\mu}(\varphi;z,r):=\int_{\CC^n}\varphi(z+ry)\,d\mu(y).
  \end{equation*}
\end{Df}

Using complex generalised integral means we are able to extend the results about homogeneous operators of first and second order.

In particular Proposition \ref{pr:first} holds in the case when $P(\partial_z)$ is a general homogeneous operator of order one with
complex variables (i.e $P(\partial_z)=\sum_{j=1}^n a_j\partial_{z_j}$ for some $a=(a_1,\dots,a_n)\in \CC^n$).

Analogously, Proposition \ref{pr:elliptic} and Theorem \ref{th:elliptic} are valid for
general homogeneous operators $P(\partial_z)$ of order two with complex coefficients
(i.e. $P(\partial_z)=\sum_{i,j=1}^n a_{ij}\partial^2_{z_i z_j}$, where $A=(a_{ij})_{i,j=1}^n$ is a symmetric nonsingular complex matrix).
In this case a diagonal matrix $\Lambda=C^TAC$ is not necessarily positive and consequently the measures $\mu_B$ and $\mu_S$ defined by
(\ref{eq:measure_el}) are supported respectively in the complex sets $C\Lambda^{1/2}(B(0,1))$ and $C\Lambda^{1/2}(S(0,1))$.

As an example we take the $n$-dimensional complex wave operator
 $$
 P(\partial_{z})=\partial_{z_1}^2-\sum_{k=2}^n\partial_{z_k}^2=\partial_{z_1}^2-\Delta_{z'}=\square_z,
 $$
 where  $z=(z_1,z_2,\dots,z_n)=(z_1,z')\in\CC^n$.
 Next we define the measure on
 $$
 B_{1}(0,1)=\{(x_1+iy_1,\dots,x_n+iy_n)\in\CC^n\colon y_1=x_2=\dots=x_n=0,\ x_1^2+y_2^2+\dots+ y_n^2\leq 1\}
 $$
 by $\mu_{B_{1}}(z)=\frac{d(z_1,-iz_2,\dots,-iz_n)}{\alpha(n)}$, 
 where $dy$ is the Lebesgue measure on $B(0,1)$,
 and the measure on
 $$
 S_{1}(0,1)=\{(x_1+iy_1,\dots,x_n+iy_n)\in\CC^n\colon y_1=x_2=\dots=x_n=0,\ x_1^2+y_2^2+\dots+ y_n^2=1\}
 $$
 by $\mu_{S_{1}}(z)=\frac{dS(z_1,-iz_2,\dots,-iz_n)}{n\alpha(n)}$,
 where $dS(y)$ is the surface measure on $S(0,1)$.
 
 Then the generalised solid $M_{\mu_{B_1}}(\varphi;z,t)$ and spherical $M_{\mu_{S_1}}(\varphi;z,t)$ means
 defined by
 \begin{equation*}
 M_{\mu_{B_1}}(\varphi;z,t)=\int_{B_{1}(0,1)}\varphi(z+ty)\,d\mu_{B_{1}}(y)=\nint{B(0,1)}\varphi(z+t(y_1,iy_2,\dots,iy_n))\,dy
 \end{equation*}
 and
 \begin{equation*}
 M_{\mu_{S_{1}}}(\varphi;z,t)=\int_{S_{1}(0,1)}\varphi(z+ty)\,d\mu_{S_{1}}(y)=\nint{S(0,1)}\varphi(z+t(y_1,iy_2,\dots,iy_n))\,dS(y)
 \end{equation*}
 satisfy Pizzetti-type formulas for the complex wave operator $\square_z$
 \begin{gather*}
 M_{\mu_{B_{1}}}(\varphi;z,t)= \sum_{j=0}^{\infty}\frac{\square_z^{j}\varphi(z)}{m_{B}(j)}t^{2j}\quad
 \textrm{and}\quad
 M_{\mu_{S_{1}}}(\varphi;z,t)=\sum_{j=0}^{\infty}\frac{\square_z^{j}\varphi(z)}{m_{S}(j)}t^{2j},
 \end{gather*}
 where $m_B(j)$ and $m_S(j)$ are defined by (\ref{eq:numbers}).
 
 So by Theorem \ref{th:sum} we get
 \begin{Cor}
 \label{co:wave} 
 Let $d\in\RR$ and $\widehat{u}=\widehat{u}(t,z)\in\Oo[[t]]$ be the formal solution
 of the initial value problem
  \begin{eqnarray*}
(\partial_t - \square_z)u = 0,&& u(0,z)=\varphi(z)\in\mathcal{O}(D^n).
\end{eqnarray*}
Then we conclude that:
\begin{itemize}
 \item $u\in\Oo(D^{n+1})$ if and only if $M_{\mu_{B_1}}(\varphi;z,t)\in\Oo^2(D^n\times \CC)$, and if and only if 
   $M_{\mu_{S_1}}(\varphi;z,t)\in\Oo^2(D^n\times \CC)$.
 \item $\widehat{u}$ is $1$-summable in the direction $d$ if and only if
   $M_{\mu_{B_1}}(\varphi;z,t)\in\Oo^{2}(D^n\times (\widehat{S}_{d/2}\cup\widehat{S}_{d/2+\pi}))$,
   and if and only if $M_{\mu_{S_{1}}}(\varphi;z,t)\in\Oo^{2}(D^n\times (\widehat{S}_{d/2}\cup\widehat{S}_{d/2+\pi}))$.
\end{itemize}
 \end{Cor}

Using complex generalised integral means we also obtain
\begin{Th}
 \label{th:s}
 Let $p,s\in\NN$. We assume that a complex generalised integral mean satisfies a Pizzetti-type formula
 $$
 M_{\mu_B}(\varphi;z,t)=\int_{B}\varphi(z+ty)\,d\mu_B(y)=\sum_{j=0}^{\infty}\frac{Q^j(\partial_z)\varphi(z)}{m(j)}t^{pj}
 $$
 for some measure $\mu_B$ supported in $B\subset\CC^n$, for some homogeneous operator $Q(\partial_z)$ of order $p$
 and for some function $m$ of order $p$.
 Moreover, we assume that the measure $\mu_B^s$ is defined by
 $$
 \mu_B^s(y):=\frac{1}{s}\sum_{k=0}^{s-1}\mu_B(e^{-\frac{2k\pi i}{ps}}y).
 $$
 Then the complex generalised integral mean $M_{\mu_B^s}(\varphi;z,t)$ satisfies a Pizzetti-type formula for the operator $Q^s(\partial_z)$
 \begin{multline*}
 M_{\mu_B^s}(\varphi;z,t)=\frac{1}{s}\sum_{k=0}^{s-1}\int_{e^{\frac{2\pi ki}{ps}}B}\varphi(z+ty)\,d\mu_B(e^{-\frac{2k\pi i}{ps}}y)\\=
 \frac{1}{s}\sum_{k=0}^{s-1}\int_{B}\varphi(z+e^{\frac{2\pi ki}{ps}}ty)\,d\mu_B(y)=
 \sum_{j=0}^{\infty}\frac{Q^{sj}(\partial_z)\varphi(z)}{m(sj)}t^{psj}.
 \end{multline*}
 Moreover, if  $\widehat{u}=\widehat{u}(t,z)\in\Oo[[t]]$
 is the formal solution of the initial value problem
  \begin{eqnarray*}
(\partial_t - Q^s(\partial_z))u = 0,&& u(0,z)=\varphi(z)\in\mathcal{O}(D^n),
\end{eqnarray*}
then we conclude that $u\in\Oo(D^{n+1})$ if and only if $M_{\mu_B^s}(\varphi;z,t)\in\Oo^{\frac{ps}{ps-1}}(D^n\times \CC)$.

If we additionally assume that $d\in\RR$ and $m$ is a generalised moment function of order $p$ then $\widehat{u}$ is $(ps-1)^{-1}$-summable in the direction $d$
if and only if
$M_{\mu_B^s}(\varphi;z,t)\in\Oo^{\frac{ps}{ps-1}}(D^n\times \widehat{S}_{\frac{d+2k\pi}{ps}})$ for $k=0,\dots,ps-1$.
 \end{Th}
 \begin{proof}
 We have
 $$
 M_{\mu_B^s}(\varphi;z,t)=\frac{1}{s}\sum_{k=0}^{s-1}\sum_{j=0}^{\infty}\frac{Q^j(\partial_z)\varphi(z)}{m(j)}(e^{\frac{2\pi ki}{s}}t^p)^j=
 \sum_{j=0}^{\infty}\frac{Q^{sj}(\partial_z)\varphi(z)}{m(sj)}t^{psj},$$ 
 which proves the first part of the theorem. The proof is completed by applying Theorem \ref{th:sum} and Remark \ref{re:m_tilde}.
\end{proof} 

 In particular, we obtain
 \begin{Cor}
   \label{co:laplace}
We assume that $s\in\NN$, $P(\partial_{z})=\Delta_z^s= (\sum_{k=1}^n\partial^2_{z_k})^s$ is the $s$-Laplace operator,
 \begin{gather*}
 \mu_B^s(y)=\frac{1}{s}\sum_{k=0}^{s-1}\mu_B(e^{-\frac{k\pi i}{s}}y)\quad\textrm{and}\quad
 \mu_S^s(y)=\frac{1}{s}\sum_{k=0}^{s-1}\mu_S(e^{-\frac{k\pi i}{s}}y),
 \end{gather*}
where the measures $\mu_B(y)$ and $\mu_S(y)$ are defined by (\ref{eq:means}).

 Then the generalised solid $M_{\mu_B^s}(\varphi;z,t)$ and spherical $M_{\mu_S^s}(\varphi;z,t)$ means satisfy 
 Pizzetti-type formulas for the operator $\Delta^s_z$
 $$
 M_{\mu_B^s}(\varphi;z,t)=\frac{1}{s}\sum_{k=0}^{s-1}\nint{B(0,1)}\varphi(z+e^{\frac{k\pi i}{s}}ty)\,dy
 =\sum_{j=0}^{\infty}\frac{\Delta_z^{sj}\varphi(z)}{m_B(sj)}t^{2sj}
 $$
 and
 $$
 M_{\mu_S^s}(\varphi;z,t)=\frac{1}{s}\sum_{k=0}^{s-1}\nint{S(0,1)}\varphi(z+e^{\frac{k\pi i}{s}}ty)\,dS(y)
 =\sum_{j=0}^{\infty}\frac{\Delta_z^{sj}\varphi(z)}{m_S(sj)}t^{2sj},
 $$
where $m_B$ and $m_S$ are given by (\ref{eq:numbers}).

Moreover, if we additionally assume that $d\in\RR$ and $\widehat{u}=\widehat{u}(t,z)\in\Oo[[t]]$
 is the formal solution of the initial value problem
  \begin{eqnarray*}
(\partial_t - \Delta_z^s)u = 0,&& u(0,z)=\varphi(z)\in\mathcal{O}(D^n),
\end{eqnarray*}
then we conclude that
\begin{itemize}
 \item $u\in\Oo(D^{n+1})$ if and only if $M_{\mu_B^s}(\varphi;z,t)\in\Oo^{\frac{2s}{2s-1}}(D^n\times \CC)$, and if and only if
$M_{\mu_S^s}(\varphi;z,t)\in\Oo^{\frac{2s}{2s-1}}(D^n\times \CC)$,
\item $\widehat{u}$ is $(2s-1)^{-1}$-summable in the direction $d$ if and only if
$M_{\mu_B^s}(\varphi;z,t)\in\Oo^{\frac{2s}{2s-1}}(D^n\times \widehat{S}_{\frac{d+2k\pi}{2s}})$ for $k=0,\dots,2s-1$, and if and only if
$M_{\mu_S^s}(\varphi;z,t)\in\Oo^{\frac{2s}{2s-1}}(D^n\times \widehat{S}_{\frac{d+2k\pi}{2s}})$ for $k=0,\dots,2s-1$.
\end{itemize}
\end{Cor}

Analogously we get
 \begin{Cor}
  \label{co:first}
  We assume that $s\in\NN$, $P(\partial_z)=(\sum_{k=1}^na_k\partial_{z_k})^s$ with $a=(a_1,\dots,a_n)\in\CC^n$ and
 $\mu_{\delta,a}^s(y)=\frac{1}{s}\sum_{k=0}^{s-1}\delta(e^{-\frac{2k\pi i}{s}}y-a)$, where $\delta$ is the Dirac delta.
 Then the complex generalised integral mean $M_{\mu_{\delta,a}^s}(\varphi;z,t)$ satisfies a Pizzetti-type formula for $P(\partial_z)$
 $$M_{\mu_{\delta,a}^s}(\varphi;z,t)=\frac{1}{s}\sum_{k=0}^{s-1}\varphi(z+e^{\frac{2k\pi i}{s}}at)
 =\sum_{j=0}^{\infty}\frac{P^{j}(\partial_z)\varphi(z)}{(sj)!}t^{sj}.$$
 If we additionally assume that $s>1$, $d\in\RR$ and $\widehat{u}=\widehat{u}(t,z)\in\Oo[[t]]$
 is the formal solution of the initial value problem
  \begin{eqnarray*}
(\partial_t - P(\partial_z))u = 0,&& u(0,z)=\varphi(z)\in\mathcal{O}(D^n),
\end{eqnarray*} 
then we conclude that
\begin{itemize}
\item $u\in\Oo(D^{n+1})$ if and only if $M_{\mu_{\delta,a}^s}(\varphi;z,t)\in\Oo^{\frac{s}{s-1}}(D^n\times \CC)$,
\item $\widehat{u}$ is $(s-1)^{-1}$-summable in the direction $d$ if and only if
$M_{\mu_{\delta,a}^s}(\varphi;z,t)\in\Oo^{\frac{s}{s-1}}(D^n\times \widehat{S}_{\frac{d+2k\pi}{s}})$ for $k=0,\dots,s-1$.
\end{itemize}
\end{Cor}
 \begin{Cor}
   \label{co:second}
   We assume that $s\in\NN$, $P(\partial_z)=(\sum_{i,j=1}^na_{ij}\partial_{z_i z_j}^2)^s$, where $A=(a_{ij})_{i,j=1}^n$ is a symmetric nonsingular
   complex matrix with the corresponding diagonal matrix $\Lambda$ (i.e. $\Lambda=C^TAC$ for some orthogonal complex matrix $C$).
   We also assume that
 \begin{gather*}
  \mu_B^s(y)=\frac{1}{s}\sum_{k=0}^{s-1}\mu_B(e^{-\frac{k\pi i}{s}}y)\quad\textrm{and}\quad
  \mu_S^s(y)=\frac{1}{s}\sum_{k=0}^{s-1}\mu_S(e^{-\frac{k\pi i}{s}}y),
 \end{gather*}
 where the measures $\mu_B$ and $\mu_S$ are defined by (\ref{eq:measure_el}).
 
 Then the complex generalised integral means $M_{\mu_B^s}(\varphi;z,t)$ and $M_{\mu_S^s}(\varphi;z,t)$ satisfy Pizzetti-type formulas for $P(\partial_z)$
 $$M_{\mu_B^s}(\varphi;z,t)=\int_{\CC^n}\varphi(z+ty)\,d\mu_B^s(y)=\sum_{j=0}^{\infty}\frac{P^{j}(\partial_z)\varphi(z)}{m_B(sj)}t^{2sj}$$
 and
 $$M_{\mu_S^s}(\varphi;z,t)=\int_{\CC^n}\varphi(z+ty)\,d\mu_S^s(y)=\sum_{j=0}^{\infty}\frac{P^{j}(\partial_z)\varphi(z)}{m_S(sj)}t^{2sj},$$
 where the sequences of numbers $m_B$ and $m_S$ are given by (\ref{eq:numbers}).
 
 If we additionally assume that $d\in\RR$ and $\widehat{u}=\widehat{u}(t,z)\in\Oo[[t]]$
 is the formal solution of the initial value problem
  \begin{eqnarray*}
(\partial_t - P(\partial_z))u = 0,&& u(0,z)=\varphi(z)\in\mathcal{O}(D^n),
\end{eqnarray*} 
then we conclude that
\begin{itemize}
\item $u\in\Oo(D^{n+1})$ if and only if $M_{\mu_B^s}(\varphi;z,t)\in\Oo^{\frac{2s}{2s-1}}(D^n\times \CC)$, and
if and only if $M_{\mu_S^s}(\varphi;z,t)\in\Oo^{\frac{2s}{2s-1}}(D^n\times \CC)$,
\item $\widehat{u}$ is $(2s-1)^{-1}$-summable in the direction $d$ if and only if
$M_{\mu_B^s}(\varphi;z,t)\in\Oo^{\frac{2s}{2s-1}}(D^n\times \widehat{S}_{\frac{d+2k\pi}{2s}})$ for $k=0,\dots,2s-1$, and
if and only if $M_{\mu_S^s}(\varphi;z,t)\in\Oo^{\frac{2s}{2s-1}}(D^n\times \widehat{S}_{\frac{d+2k\pi}{2s}})$ for $k=0,\dots,2s-1$.
\end{itemize}
  \end{Cor}

\section{Final remarks}
In the presented paper, we are concerned with the application of the Pizzetti-type formulas to the characterisation of the convergent and summable formal power series
solutions of the equation (\ref{eq:1.1}) in terms of holomorphic properties of generalised integral means of the Cauchy data.

Another, maybe more trivial, application is indicated in Remark \ref{re:integral_mean}.
Namely, immediately by the Pizzetti type formula we get the mean-value property and
converse of the mean value property for a given operator $P(\partial_z)$.

The Pizzetti formula has also the applications in such different fields of interest as asymptotic solutions of PDEs, spherical harmonics, and numerical analysis.
For the deeper discussion about these applications we refer the reader to the introduction of \cite{Bo2} and the references therein.

We hope that new examples of Pizzetti type formulas in the present paper, especially Corollaries \ref{co:laplace}--\ref{co:second},
will allow to obtain new applications in the mentioned above fields of interest.

In particular, in author's opinion, one can apply the Pizzetti type formula for
the $s$-Laplace operator $\Delta^s$, given in Corollary \ref{co:laplace}, to the study of $s$-polyharmonic generalisations of spherical harmonics and zonal functions
in the spirit similar to \cite{Be-D-S}.

Moreover, using the ideas from Theorem \ref{th:s}, it seems to be possible to generalise the Pizzetti type expansion
for the mean value over a geodesic sphere in a Riemannian manifold from \cite{G-W} to an $s$-complex version of a geodesic sphere.

We are going to study these problems in the subsequent papers.

\subsection*{Acknowledgement}
The author would like to thank the anonymous referee for valuable comments and suggestions.

    \bibliographystyle{siam}
    \bibliography{summa}
    \end{document}